\documentclass{amsart}
\usepackage{latexsym}
\usepackage{amsmath}
\usepackage{amssymb}
\usepackage[all,cmtip,2cell]{xy}
\UseTwocells

\newtheorem{thm}{Theorem}[section]
\newtheorem{prop}[thm]{Proposition}

\newtheorem{cor}[thm]{Corollary}
\newtheorem{lem}[thm]{Lemma}
\theoremstyle{definition}
\newtheorem{defn}[thm]{Definition}

\newtheorem{assertion}[thm]{Assertion}

\theoremstyle{remark}
\newtheorem{rem}[thm]{Remark}
\newtheorem{ex}[thm]{Example}

\newcommand{\G}{{\mathcal G}}
\newcommand{\C}{{\mathcal C}}
\newcommand{\E}{{\mathcal E}}
\newcommand{\calH}{{\mathcal H}}
\newcommand{\e}{\varepsilon}
\newcommand{\D}{\mathcal{D}}
\newcommand{\mapright}[1]{%
 \smash{\mathop{%
  \hbox to 1cm{\rightarrowfill}}\limits_{#1}}}
\newcommand{\maprightd}[2]{%
 \smash{\mathop{%
  \hbox to 1.2cm{\rightarrowfill}}\limits^{#1}\limits_{#2}}}
\newcommand{\mapleft}[1]{%
 \smash{\mathop{%
  \hbox to 1cm{\leftarrowfill}}\limits_{#1}}}
\newcommand{\mapleftu}[1]{%
 \smash{\mathop{%
  \hbox to 0.8cm{\leftarrowfill}}\limits^{#1}}}
\newcommand{\maprightu}[1]{%
 \smash{\mathop{%
  \hbox to 1cm{\rightarrowfill}}\limits^{#1}}}
\newcommand{\maprightud}[2]{%
 \smash{\mathop{%
  \hbox to 1cm{\rightarrowfill}}\limits^{#1}_{#2}}}
\newcommand{\mapleftud}[2]{%
 \smash{\mathop{%
  \hbox to 1cm{\leftarrowfill}}\limits^{#1}_{#2}}}

\newcounter{eqn}[section]

\def\theeqn{\textnormal{(\thesection.\arabic{eqn})}}

\def\eqnlabel#1{%
  \refstepcounter{eqn}%
  \label{#1}%
  \leqno{\theeqn}}

\begin{document}

\title[On strong homotopy for quasi-schemoids]{
On strong homotopy for quasi-schemoids \\
}

\footnote[0]{

This research was partially supported by a Grant-in-Aid for Scientific
Research HOUGA 25610002
from Japan Society for the Promotion of Science.

Department of Mathematical Sciences, 
Faculty of Science,  
Shinshu University,   
Matsumoto, Nagano 390-8621, Japan   
e-mail:{\tt kuri@math.shinshu-u.ac.jp} 
}

\author{Katsuhiko KURIBAYASHI}
\date{}
   
\maketitle

\begin{abstract}
A quasi-schemoid is a small category with a particular partition of the set of morphisms. 
We define a homotopy relation on the category of quasi-schemoids and study 
its fundamental properties.  
As a homotopy invariant, the homotopy set of self-homotopy equivalences on a quasi-schemoid is introduced.  
The main theorem enables us to deduce that the homotopy invariant for the quasi-schemoid induced by a finite group 
is isomorphic to the automorphism group of the given group.  
\end{abstract}

\section{Introduction} The category $\mathsf{Cat}$ of small categories is a $2$-category whose $2$-morphisms are natural transformations. 
Hoff \cite{Hoff} and Lee \cite{L} have introduced a notion of strong homotopy 
on $\mathsf{Cat}$ using $2$-morphisms; see also \cite{Hoff_1975, L-T-W, M}. 
Thus if the objects we investigate 
have the structure of small categories, we may develop homotopy theory for them with the underlying small categories.   

Association schemes play crucial roles in the study of algebraic combinatorial theory, design and coding theory; 
see for example \cite{P-Z} and 
references contained therein. Very recently, such combinatorial objects were used in investigating 
continuous-time quantum walks from a mathematical perspective; see \cite{G, G-H-R}. 
This motivates us to consider their classification problem.  
Though it is important to classify such subjects in the strict sense \cite{H-M}, namely up to isomorphism,  
one might make a rough classification of association schemes relying on abstract homotopy theory. 
Since association schemes can be regarded as complete graphs, and hence objects in $\mathsf{Cat}$,  
the completeness allows us to deduce that every association scheme is contractible in the sense of strong homotopy.
In fact,  each association scheme is equivalent to the trivial category as a category. 
Thus we need an appropriate category instead of $\mathsf{Cat}$ in which  
to develop meaningful homotopy theory for combinatorial objects such as association schemes.   


Matsuo and the author \cite{K-M} have proposed the notion of {\it quasi-schemoids} generalizing that of association schemes from 
a small categorical point of view. 
Roughly speaking, the new object is indeed a small category with suitable coloring for morphisms. 
In this paper, we define a homotopy relation on the category $q\mathsf{ASmd}$ of quasi-schemoids extending that due to Hoff, Lee and Minian,  
and study the fundamental properties of homotopy. 
In particular, the group (of homotopy classes) of self-homotopy equivalences on a quasi-schemoid is investigated.    

An important point here is that $q\mathsf{ASmd}$ admits a $2$-category structure under which 
the category $\mathsf{Cat}$ is embedded into the category $q\mathsf{ASmd}$ as a $2$-category; 
see Theorem \ref{thm:2-cat} below. Thus one might expect a relevant notion of 
a homotopy group  for a quasi-schemoid, as in 
\cite{Hoff_1975}, and an application of categorical matrix Toda brackets due to Hardie, Kamps and Marcum \cite{H-K-M} to our category 
$q\mathsf{ASmd}$. As for homological algebra on schemoids, in order to develop categorical representation theory, 
we may consider the Bose-Mesner algebra introduced in \cite[Section 2]{K-M} and 
an appropriate functor category with a quasi-schemoid and an abelian category as source and target, respectively; see \cite[Sections 5 and 6]{K-M} for first steps in this direction.   
These topics are addressed in subsequent work. 

Though {\it association schemoids} and their category $\mathsf{ASmd}$ are also introduced in \cite{K-M}, 
we do not develop homotopy theory in $\mathsf{ASmd}$ in this paper; see the Appendix. 

This manuscript is organized as follows.
In Section 2, we recall the definition of a quasi-schemoid with examples. 
Section 3 explains a homotopy relation which we use in the category of quasi-schenoids. 
Section 4 is devoted to describing rigidity properties of  homotopy for association schemes and groupoids. 
In particular, our main theorem (Theorem \ref{thm:haut(groupoid)}) asserts that  the group of self-homotopy 
equivalences on the quasi-schemoid arising from a groupoid includes  
the group of autofunctors on the given groupoid. 
It turns out that the group of self-homotopy equivalences on a finite group 
is isomorphic to the automorphism group of the given group.

\section{A brief review of quasi-schemoids}

We begin by recalling the definition of an association scheme. 
Let $X$ be a finite set and $S$ a partition of the Cartesian square $X\times X$, namely a subset of the power set $2^{X\times X}$ with 
$X\times X = \amalg_{\sigma\in S}\sigma$, 
which contains the subset $1_X :=\{ (x, x) \mid x \in X\}$ as an element. 
Assume further that for each 
$g \in S$, the subset $g^*:=\{(y, x) \mid (x, y) \in g\}$ is in $S$. Then the pair $(X, S)$ is called an  
{\it association scheme}  if for all 
$e, f, g \in S$, there exists an integer $p_{ef}^g$ such that for any $(x, z) \in g$,  
$$
p_{ef}^g=\sharp \{y \in X \mid (x, y)\in e \ \text{and} \ (y, z) \in f \}. 
$$
Observe that $p_{ef}^g$ is independent of the choice of $(x, z) \in g$. 

Let $G$ be a finite group. Define a subset $G_f$ of $G\times G$ for $f\in G$ by 
$G_f:=\{(k, l) \mid k^{-1}l = f\}$. Then we have an association scheme $S(G) = (G, [G])$, where 
$[G]=\{G_f\}_{f\in G}$. 
Moreover, the correspondence $S( \ )$ induces a functor from the category $\mathsf{Gr}$ of finite groups to the category  
$\mathsf{AS}$ of association schemes in the sense of Hanaki \cite{H}; see also \cite[Section 5.5]{Z_book}.

We here recall the definition of a quasi-schemoid, which is a categorical counterpart of an association scheme.  

\begin{defn}(\cite[Definition 2.1]{K-M})\label{defn:schemoid}
Let $\C$ be a small category; that is, the class of the objects of the category $\C$ is a set. 
Let $S:=\{\sigma_l\}_{l\in I}$ be a partition of the set $mor(\C)$
of all morphisms in $\C$. We call the pair $(\C, S)$ a {\it quasi-schemoid} if 
the set $S$ satisfies the condition that for a triple $\sigma, \tau, \mu \in S$ 
and for any morphisms $f$, $g$ in $\mu$, as a set 
$$
(\pi_{\sigma\tau}^\mu)^{-1}(f) \cong (\pi_{\sigma\tau}^\mu)^{-1}(g), 
$$ 
where $\pi_{\sigma\tau}^\mu : \pi_{\sigma\tau}^{-1}(\mu) \to \mu$ denotes 
the restriction of the concatenation map 
$\pi_{\sigma\tau} : \sigma \times_{ob(\C)}\tau:=\{(f, g) \in \sigma \times \tau \mid s(f) = t(g)\} \to mor(\C)$.
\end{defn}

We denote by $p_{\sigma\tau}^\mu$ the cardinality of the set $(\pi_{\sigma\tau}^\mu)^{-1}(f)$. 

For an association scheme $(X, S)$, we define a quasi-schemoid $\jmath(X, S)$ by the pair $(\C, V)$ for which 
$ob(\C)=X$, $\text{Hom}_\C(y, x) =\{(x, y)\} \subset X\times X$ and $V =S$, where the composite of morphisms 
$(z, x)$ and $(x, y)$ is defined by $(z, x) \circ (x, y) = (z, y)$.

For a groupoid ${\mathcal H}$, we have 
a quasi-schemoid $\widetilde{S}({\mathcal H}) = (\widetilde{\mathcal H}, S)$ , where $ob (\widetilde{\mathcal H}) = mor({\mathcal H})$ 
and 
$$
\text{Hom}_{\widetilde{{\mathcal H}}}(g, h) = \begin{cases}
\{(h, g)\}  & \text{if}  \ \  t(h) = t(g)  \\
\varnothing & \text{otherwise}. 
\end{cases}
$$
The partition $S=\{ \G_f \}_{f \in mor({\mathcal H})}$ is defined by $\G_f = \{(k, l) \mid k^{-1}l = f\}$. 
We refer the reader to \cite[Section 2]{K-M} for more examples of quasi-schemoids.  

Let $(\C, S)$ and $(\E, S')$ be quasi-schemoids. It is readily seen that  $(\C\times \E, S\times S')$ is a quasi-schemoid, where 
$S\times S' =\{\sigma\times \tau \mid \sigma \in S, \tau \in S'\} \subset mor(\C)\times mor(\E)$. In what follows, we write 
$(\C, S)\times (\E, S')$ for the product. 

\begin{defn}\label{defn:morphisms}
Let $(\C, S)$ and $(\E, S')$ be quasi-schemoids. A functor $F: \C \to \E$ is a {\it morphism} of quasi-schemoids 
if for any $\sigma$ in $S$, 
$F(\sigma) \subset \tau$ for some $\tau$ in $S'$. 
We then write $F : (\C, S) \to (\E, S')$ for the morphism.  
\end{defn}

We denote by $q\mathsf{ASmd}$ the category of quasi-schemoids and their morphisms. 
Let $\C$ be a small category and  $K(\C)=(\C, S)$ 
the discrete quasi-schemoid associated with $\C$; that is, the partition $S$ 
is defined by $S=\{\{f\}\}_{f \in mor(\C)}$.  
The correspondence $K$ induces a pair of adjoints  
$
\xymatrix@C17pt@R10pt{
K :  \mathsf{Cat}\ar@<0.5ex>[r]^-{} &q\mathsf{ASmd}  : U  \ar@<0.5ex>[l]^-{} 
 }
$
in which $U$ is the forgetful functor and the right adjoint to $K$.
It is remarkable that the functor $K$ is a fully faithful embedding; see \cite[Remark 3.1, Diagram (6.1)]{K-M}. 
Furthermore, the correspondences $\widetilde{S}( \ )$ and $\jmath$ mentioned above give rise to functors. With such functors,  
we obtain a commutative diagram of categories 
$$
\xymatrix@C35pt@R20pt{
\mathsf{Gpd } \ar[r]^-{{\widetilde S}( \ )} & 
q\mathsf{ASmd} \ar@<1ex>[r]^-{U} 
& \mathsf{Cat}, \ar@<1ex>[l]^-{K}  \\
\mathsf{Gr} \ar[u]^\imath \ar[r]^-{S( \ )}  & \mathsf{AS} \ar[u]_{\jmath} 
}
\eqnlabel{add-1}
$$ 
where $\mathsf{Gpd}$ denotes the category of groupoids and $\imath : \mathsf{Gr} \to \mathsf{Gpd}$ is the natural fully faithful embedding;  
see \cite[Section 3, Diagram (6.1)]{K-M}. Observe that the composite $U\circ {\widetilde S}( \ )$ is {\it not} the usual embedding 
from $\mathsf{Gpd}$ to $\mathsf{Cat}$. 

The homotopy category of $\mathsf{Cat}$ in the sense of Thomason is equivalent to that of topological spaces 
\cite{Latch, Thomason}. Moreover, a result of \cite[Theorem 3.2]{K-M} asserts that the functors 
$S( \ )$ and $\widetilde{S}( \ )$ are  faithful and that $\jmath$ is a fully faithful embedding. 
Thus quasi-schemoids can be regarded as {\it generalized spaces} and as {\it generalized groups} in some sense.

\section{Strong homotopy}

We extend the notion of strong homotopy in $\mathsf{Cat}$ in the sense of Hoff \cite{Hoff_1975} and Lee \cite{L}  to that in $q\mathsf{ASmd}$.  
Let $[1]$ be the category consisting of two objects $0$ and $1$ and only one non-trivial morphism 
$u : 0 \to 1$. We write $I$ for a discrete schemoid of the form $K([1])$. 

\begin{defn}\label{defn:Homotopy}
Let $F, G : (\C, S) \to (\D, S')$ be morphisms between the schemoids $(\C, S)$ and  $(\D, S')$ in $q\mathsf{ASmd}$. 
We write $H: F \Rightarrow G$ if there exists a morphism $H : (\C, S) \times I \to (\D, S')$ in  $q\mathsf{ASmd}$ such that 
$H\circ \e_0 = F$ and $H\circ \e_1 = G$, where $(\C, S) \times I$ is the product of the quasi-schemoids mentioned in Section 2
and 
$\e_i : (\C, S) \to (\C, S) \times I$ is the morphism of quasi-schemoids defined by 
$\e_i(a) = (a, i)$ for an object $a$ in $\C$ and $\e_i(f)= (f,  1_i)$ for a morphism $f$ in $\C$.  
We call the morphism $H$ above a {\it homotopy} from $F$ to $G$. 

A morphism $F$ is {\it equivalent} to $G$, 
denoted $F\sim G$, if $H: F \Rightarrow G$ or $H: G \Rightarrow F$ for some 
$H : (\C, S) \times I \to (\D, S')$ in  $q\mathsf{ASmd}$. 
\end{defn}

\begin{rem}\label{rem:square} Suppose that there exists a homotopy $H : (\C, S) \times I \to (\D, S')$ from 
$F$ to $G$. Then for any morphism $f \in mor(\C)$, we have a commutative diagram 
$$
\xymatrix@C35pt@R20pt{
{H(s(f), 0)}  \ar@{->}[r]^{H(1_{s(f)}, u)} \ar[dr]^{H(f,u)}   \ar@{->}[d]_{F(f)=H(f, 1_0)}  &  {H(s(f), 1)}  \ar@{->}[d]^{H(f, 1_1)=G(f)}\\
{H(t(f), 0)} \ar@{->}[r]_{H(1_{t(f)}, u)}  &  {H(s(f), 1)}
}
$$
in the underlying category $\D$. Here we use the same notation as in Definition \ref{defn:Homotopy}. 

Since $H$ is a morphism of quasi-schemoids, it follows that $H(g, u)$ and $H(h, u)$ are in the same element of $S'$
if $g$ and $h$ are in the same element of $S$. 
We observe that, in {\it each} square for a given morphism $f$, morphisms $H(1_{s(f)}, u)$ and $H(1_{t(f)}, u)$ are 
in the same element of $S'$ if $1_{s(f)}$ and $1_{t(f)}$ are in the same element of $S$. 
In fact, the condition is satisfied if the quasi-schemoid comes from an association scheme. 
As for the diagonal arrows,  in order to show the well-definedness of the homotopy $H$ in 
$q\mathsf{ASmd}$, we need to verify that 
the arrow $H(f, u)$ in a square and $H(g, u)$ in {\it other} squares are in the same element of $S'$  
if $g$ is in the same element of $S$ as that containing $f$.    
\end{rem}

In what follows, we will define a homotopy assigning objects and morphisms 
in $\D$ to those in $\C \times I$ as in the square above.   

Let $F: (\C, S) \to (\D, S')$ be a morphism of quasi-schemoids. Then for any $f : i \to j$ in $mor(\C)$, 
we have a commutative diagram 
$$
\xymatrix@C35pt@R20pt{
{F(i)}  \ar@{->}[r]^{F(1_i)} \ar[dr]^{F(f)}   \ar@{->}[d]_{F(f)}  &  {F(i)}  \ar@{->}[d]^{F(f)}\\
{F(j)} \ar@{->}[r]_{F(1_j)}  &  {F(j)}
}
$$
in the underlying category $\D$. If $1_i$ and $1_j$ are in the same element of $S$,   
$F(1_i)$ and $F(1_j)$ are in the same element of $S'$. 
The diagram gives rise to a homotopy from $F$ to itself. 
 
\begin{defn} \label{defn:Homotopy_relation} Let $(\C, S)$ and $(\D, S')$ be a quasi-schemoids.  
For morphisms $F, G : (\C, S) \to (\D, S')$, $F$ is {\it homotopic} to $G$, denoted $F \simeq G$, if there exists a finite sequence of morphisms 
$F= F_0, F_1, ..., F_n= G$ such that  $F_k \sim F_{k+1}$ for any $k = 0, ..., n$.  
We say that $(\C, S)$ is {\it homotopy equivalent} to $(\D, S')$ if there exist morphisms $F : (\C, S) \to (\D, S')$ and $G :  (\D, S') \to  (\C, S)$ such that $FG \simeq 1$ and $GF\simeq 1$. In this case, $F$ is called 
a {\it homotopy equivalence}.  
\end{defn}

The usual argument gives the following result.

\begin{prop} \label{prop:equivalence_relations}
The homotopy relation $\simeq$ in the category $q\mathsf{ASmd}$ defined in  Definition \ref{defn:Homotopy_relation} is an equivalence 
relation which is preserved by compositions of morphisms. 
\end{prop}

We denote by $\simeq_S$ the homotopy relation, which is called strong homotopy, 
in the category $\mathsf{Cat}$ due to Hoff \cite{Hoff_1975}, Lee \cite{L} 
and Minian \cite{M}.  The relation is defined in 
the same way  as in Definitions \ref{defn:Homotopy} and \ref{defn:Homotopy_relation}.

\begin{prop}\label{prop:homotopy}
Let $F, G : (\C, S) \to (\D, S')$ be morphisms in $q\mathsf{ASmd}$. Then $U(F) \simeq_S U(G)$ if $F\simeq G$. 
Assume further that $(\C, S) = K(\C)$, namely a discrete schemoid. 
Then $F\simeq G$ if and only if $U(F) \simeq_S U(G)$. 
\end{prop}

\begin{proof}
Let $H$ be a homotopy between $F$ and $G$. Since $U((\C, S)\times I) = U((\C, S))\times [1]$, it follows that $U(H)$ is a homotopy between 
$U(F)$ and $U(G)$. We have the first of the results. 

Suppose that $(\C, S)$ is the discrete quasi-schemoid $K(\C)$. The forgetful functor $U$ gives rise to 
a natural bijection  
\begin{eqnarray*}
U : \text{Hom}_{q\mathsf{ASmd}}(K(\C)\times I, (\D, S')) &=& \text{Hom}_{q\mathsf{ASmd}}(K(\C \times [1]), (\D, S')) \\
&\stackrel{\cong}{\to}&
\text{Hom}_{\mathsf{Cat}}(\C \times [1], U((\D, S'))).
\end{eqnarray*}
This implies that $L : \C\times [1] \to  U((\D, S'))$ is a homotopy from $U(F)$ to $U(G)$ 
if $U^{-1}(L)$ is a homotopy from $F$ to $G$. We have the result. 
\end{proof}

Let $\text{aut}((\C, S))$ denote the monoid of self-homotopy equivalences on $(\C, S)$ in $q\mathsf{ASmd}$; that is, 
the composition of the equivalences gives rise to the product in the monoid. 
Then the monoid structure gives a group structure on 
the set of equivalence classes 
$$
h\text{aut}((\C, S)) := \text{aut}((\C, S))/\simeq . 
$$  
We observe that the group $h\text{aut}((\C, S))$ is a homotopy invariant for quasi-schemoids. 

Proposition \ref{prop:homotopy} enables us to deduce that the functor 
$U$ induces a map 
$$
\widetilde{U} : [(\C, S), (\D, S')]:= \text{Hom}_{q\mathsf{ASmd}}((\C, S), (\D, S'))/\simeq \ \longrightarrow
\text{Hom}_{\mathsf{Cat}}(\C, \D)/\simeq_S 
$$
which is a bijection provided  $(\C, S)$ is a discrete quasi-schemoid. In particular, the homomorphism of groups 
$
\widetilde{U} : h\text{aut}(K(\C)) \longrightarrow h\text{aut}(\C)
$
is an isomorphism. Moreover, 
the composition of morphisms in $q\mathsf{ASmd}$ gives rise to a left  $h\text{aut}((\D, S))$-set structure and 
a right  $h\text{aut}((\C, S))$-set structure on the homotopy set $[(\C, S), (\D, S')]$. This follows from Proposition 
\ref{prop:equivalence_relations}. 

Let $B: \mathsf{Cat} \to \mathsf{Top}$ be the functor which sends a small category to its classifying space. 
A natural transformation between functors $F$ and $G$ induces a homotopy between $BF$ and $BG$. 
This enables us to conclude that  
$B\circ U$ induces a group homomorphism 
$$
\rho : h\text{aut}((\C, S)) \longrightarrow \E(B\C), 
$$
where $\E(X)$ denotes the homotopy set of self-homotopy equivalences on a space $X$. 

We here give an example of a contractible quasi-schemoid. 
Let $\C$ be a small category in which $\sigma:=\{\phi_{ij} : i \to j\}_{i, j \in ob(\C)}$ is the set of non-identity morphisms and 
the composite is given by  
$\phi_{jk}\circ\phi_{ij}= \phi_{ik}$.   Let ${\bf 1}$ be the set of all identity maps in $\C$. 
Then it follows that $(\C, S=\{\sigma, {\bf 1}\})$ is a quasi-schemoid. In fact, it is readily seen that   
$p_{{\bf 1} \sigma}^\sigma=1$, $p_{\sigma {\bf 1}}^\sigma=1$, $p_{{\bf 1} {\bf 1}}^\sigma=0$, 
$p_{{\bf 1} {\bf 1}}^{\bf 1}=1$, $p_{{\bf 1} \sigma}^{\bf 1}=0$, 
$p_{\sigma {\bf 1}}^{\bf 1}=0$ and $p_{\sigma \sigma}^{\bf 1}=0$. Moreover, 
we see that the map $\theta : (\pi_{\sigma\sigma}^\sigma)^{-1}(\phi_{ij}) \to ob(\C)$ 
defined by $\theta((\phi_{kj}, \phi_{ik}))=k$ is bijective. 

Let $\bullet$ be the trivial category; that is, it consists of one object $\bullet$ and the identity. 
We call the quasi-schemoid $K(\bullet)$ the trivial schemoid. 

\begin{prop}
The schemoid $(\C, S=\{\sigma, {\bf 1}\})$ mentioned above is contractible; that is, 
it is homotopy equivalent to the trivial schemoid.  
\end{prop}

\begin{proof}
Let $0$ be an object of $\C$. We define a morphism $s : K(\bullet) \to (\C, S)$ in $q\mathsf{ASmd}$ 
by $s(\bullet) = 0$. Let  $p:  (\C, S) \to K(\bullet)$ be the trivial morphism. 
We define a homotopy $H : (\C, S) \times I \to (\C, S)$ by 
$$
\xymatrix@C25pt@R20pt{
{k}  \ar@{->}[r]^{\phi_{k0}}  \ar[dr]^{\phi_{k0}} \ar@{->}[d]_{\phi_{kl}} &  {0}  \ar@{->}[d]^{id_0}\\
{l} \ar@{->}[r]_{\phi_{l0}} &  {0} }
$$
for any $\phi_{kl}$. Observe that $\phi_{k0}$ and $\phi_{l0}$ are in $\sigma$ for any $k$ and $l$. 
Thus we see that $1_\C \sim sp$.  We have the result.
\end{proof}

The following proposition gives a sufficient condition for 
a quasi-schemoid $(\C, S)$ not to be contractible.

\begin{prop}\label{prop:mono1}
Let $F: (\C, S) \to (\C, S)$ be a morphism of quasi-schemoids 
which is homotopic to the identity functor.
Suppose that  ${\bf 1}=\{1_x\}_{x\in ob(\C)}$ is a subset of an element in the partition $S$ 
and that $F(f)$ is an identity for some non-identity element $f \in mor(\C)$.   
Then there exist elements $\sigma$ and $\tau$ such that $\tau$ contains a non-identity element and 
$p_{\sigma\tau}^\sigma\neq 0$ or 
$p_{\tau\sigma}^\sigma\neq 0$.  
\end{prop}

\begin{proof} By assumption, we have a sequence of morphisms 
$F = F_0\sim F_1 \sim \cdots \sim F_{n-1}\sim F_n= 1_\C$. Since $F(f)$ is an identity but not $f$, 
there exists a number $l$ such that 
$F_l(f)$ is an identity and $F_{l+1}(f)$ is not an identity. Then the homotopy $H$ which induces 
the relation $F_i \sim F_{i+1}$ gives rise to a commutative diagram

\ \ \ \ \ \ $
\xymatrix@C25pt@R20pt{
{sF_l(f)}  \ar@{->}[r]^{\phi} \ar@{->}[d]_{F_l(f)=1} &  {sF_{l+1}(f)}  \ar@{->}[d]^{F_{l+1}(f)}     \\
{tF_l(f)} \ar@{->}[r]_{\phi'} &  {tF_{l+1}(f)} 
}
$
or \ 
$
\xymatrix@C25pt@R20pt{
{sF_l(f)}   \ar@{->}[d]_{F_l(f)=1} &  {sF_{l+1}(f)} \ar@{->}[l]_{\phi} \ar@{->}[d]^{F_{l+1}(f)}     \\
{tF_l(f)} &  {tF_{l+1}(f)}. \ar@{->}[l]^{\phi'} 
}
$

\noindent
Since ${\bf 1}$ is a subset of an element in $S$, it follows that $\phi$ and $\phi'$ are in the same element $\sigma$ in the partition $S$; 
see Remark \ref{rem:square}. 
We choose an element $\tau$ in $S$ which contains the morphism $F_{l+1}(f)$. It turns out that 
$p_{\sigma\tau}^\sigma\neq 0$ or 
$p_{\tau\sigma}^\sigma\neq 0$.  
\end{proof}

\begin{rem}
Let us consider a quasi-schemoid $(\C, S)$ whose underlying category $\C$ is defined by the diagram 
$$
\xymatrix@C35pt@R8pt{
 & a \ar[rd]^{\beta} & \\ 
x \ar[rr]^{\varepsilon} \ar[ru]^{\alpha} \ar[rd]_{\gamma} & & y & \text{with} \ \  \beta \alpha = \varepsilon 
= \delta\gamma  \\
 & b  \ar[ru]_{\delta} & 
}
$$
and whose partition $S=\{\sigma_1, \sigma_2, \sigma_3, {\bf 1}\}$ of $mor(\C)$ is given by 
$\sigma_1=\{\alpha, \gamma\}$, 
$\sigma_2 =\{\beta, \delta\}$, $\sigma_3=\{\varepsilon\}$ and ${\bf 1}=\{1_x, 1_y, 1_{a}, 1_{b}\}$.  
A direct computation enables us to deduce that $p_{\sigma\tau}^\sigma = 0$ and  
$p_{\tau\sigma}^\sigma = 0$ for $\sigma, \tau \in S$ if $\tau \neq {\bf 1}$.  
Then Proposition \ref{prop:mono1} implies that the quasi-schemoid $(\C, S)$ is not contractible in $q\mathsf{ASmd}$.
We observe that the underlying category $U(\C, S)=\C$ 
is contractible in $\mathsf{Cat}$ because $\C$ has an initial (terminal) object; see \cite[(3.7) Proposition]{L-T-W}. 
\end{rem}

We conclude this section after describing a $2$-category structure on $q\mathsf{ASmd}$. 

Let $I_m$ be a discrete quasi-schemoid of the form $K([m])$. For morphisms $F$ and $G$ from $(\C, S)$ to $(\D, S')$, if 
there exists a non-negative integer $m$ and a morphism $\phi : (\C, S) \times I_m \to (\D, S')$ such that 
$\phi\circ \e_0 = F$ and $\phi\circ \e_m = G$, then we write $\phi : F \Rightarrow_m G$ or 
$
\xymatrix@C+1.5pc{
(\C, S) \rtwocell^{F}_{G}{\hspace*{-0.5cm}m \;\;\;\;\  \phi} & (\D, S')
}
$
when emphasizing the source and target of the functors.  We call such a morphism $\phi$ a {\it homotopy} from 
$F$ to $G$. Observe that there exists a homotopy $\phi : F \Rightarrow_m G$ if and only if 
$\phi_0 : F \Rightarrow F_1$, $\phi_{1} : F_{1} \Rightarrow F_2$, ..., $\phi_{m-1} : F_{m-1} \Rightarrow G$ 
for some functors $F_i$ and morphisms $\phi_j$; 
see Definition \ref{defn:Homotopy}. Then we identify $\phi$ with the composite $\phi_{m-1} \circ \cdots \circ \phi_0$.


\begin{thm}\label{thm:2-cat} 
The category $q\mathsf{ASmd}$ of quasi-schemoids admits a $2$-category structure whose $2$-morphisms are homotopies mentioned above and under which the fully faithful embedding $K : \mathsf{Cat} \to q\mathsf{ASmd}$ is a functor of $2$-categories. 
\end{thm}

\begin{proof} Let $(\C, S)$ and $(\D, S')$ be quasi-schemoids. 
We then see that the hom-set  
$${\mathcal A}((\C, S), (\D, S')):=\text{Hom}_{q\mathsf{ASmd}}((\C, S), (\D, S'))$$ is a category whose objects are morphisms from $(\C, S)$ to $(\D, S')$ in $q\mathsf{ASmd}$ and morphisms are homotopies 
between them.  
Observe that the composite $\psi\circ \phi : F \Rightarrow_{m+n} L$ of  two homotopies $\phi : F \Rightarrow_m G$ and 
$\psi : G \Rightarrow_n L$ is the vertical composite of natural transformations. Moreover,  the interchange law in $\mathsf{Cat}$ enables us to deduce that the horizontal composition of the homotopies 
\begin{center}
$\xymatrix@C+1.5pc{
(\C, S) \rtwocell^{F_1}_{F_2}{\hspace*{-0.5cm}m \;\;\;\;\  \kappa} & (\D, S')
}$  \ \ and \ \  
$\xymatrix@C+1.5pc{
(\D, S') \rtwocell^{G_1}_{G_2}{\hspace*{-0.5cm}n \;\;\;\;\  \nu} & (\E, S'')
}$  
\end{center}
gives rise to a functor 
$
\ast : {\mathcal A}((\D, S'), (\E, S'')) \times {\mathcal A}((\C, S), (\D, S')) \to {\mathcal A}((\C, S), (\E, S'')).
$
In fact, the composite $\nu \ast \kappa$ is defined to be the vertical composite $(\nu F_2)\circ (G_1\kappa)$ 
of natural transformations, which coincides with the vertical composite $(G_2\kappa)\circ (\nu F_1)$. 

To prove the theorem, it suffices to show the well-definedness of the horizontal composition. 
Suppose that $\nu : G_1  \Rightarrow_1 G_2$ is a homotopy in the sense of Definition  \ref{defn:Homotopy}. 
Since $F_2$ preserves the partition, it follows from Remark  \ref{rem:square} that 
$\nu F_2 : G_1F_2 \Rightarrow G_2F_2$ is a well-defined homotopy in $q\mathsf{ASmd}$.  
Thus for any $\nu : G_1  \Rightarrow_n G_2$, in general, 
$\nu F_2$ is the composite of homotopies in the sense of Definition  \ref{defn:Homotopy}. 
The same argument yields that $G_1\kappa$ is the composite of homotopies and hence so is 
$\nu \ast \kappa$. It turns out that $\ast$ is well defined. 
\end{proof}

\section{Rigidity of homotopy for trivial association schemes and groupoids}

We first investigate the structure of the group of self-homotopy equivalences 
on a trivial association scheme.  

\begin{lem}\label{lem:iso}
Let $(X, S)$ 
be an association scheme with the trivial partition $S=\{{\bf 1}, \sigma\}$. 
Then every self-homotopy equivalence on $\jmath(X, S)$ is an isomorphism.
\end{lem}

\begin{proof} The assertion is trivial if $\sharp X =1$. Assume that $\sharp X \geq 2$. 
Let $F$ be a self-homotopy equivalence on $\jmath(X, S)$. 
We have a sequence of morphisms 
$GF \sim F_1 \sim \cdots \sim F_{n}\sim  1_\C$, where $G$ is a homotopy inverse of $F$. 
Then there exists an integer $l$ such that $F_{l+1}$ is injective and hence bijective on $X$ but not $F_l$. 
Suppose that $F_l(i)=x=F_l(j)$ for some distinct elements $i$ and $j$ of $X$. Since $F_l(\phi_{ij})= 1_x$ and $F_l$ is a  
morphism of schemoids, it follows that $F_l(f) = 1_x$ for any $f \in mor(\jmath(X, S))$. 
In fact, we see that $F_l(\phi_{ij}\circ \phi_{t(f)i}\circ f)= F_l(\phi_{ij})\circ F_l(\phi_{t(f)i})\circ F_l(f) = 1_x\circ 1_z \circ 1_y$ for some 
$z$ and $y$ in $X$. Then $x= z= x$.  

Let $H$ be a homotopy between 
$F_l$ and $F_{l+1}$, say $H: F_l  \Rightarrow F_{l+1}$. We choose an object $j'$ with $F_{l+1}(j')=x$.    
Then for a map $f : i' \to j'$ which is not the identity,  the homotopy $H$ gives a commutative diagram 
$$
\xymatrix@C40pt@R15pt{
{x}  \ar@{->}[r]^{\phi_{xF_{l+1}(i')}} \ar@{->}[d]_{1_x} &  {F_{l+1}(i')}  \ar@{->}[d]^{F_{l+1}(f)} \\
{x} \ar@{->}[r]_{\phi_{xx}=1_x} &  x. }
$$
We see that $\phi_{xF_{l+1}(i')}$ is in ${\bf 1} \in S$ and hence $F_{l+1}(i')=x$, which is a contradiction. 
This completes the proof. 
\end{proof}

\begin{rem}
An association scheme with the trivial partition is not contractible in general. 
\end{rem}

\begin{lem}\label{lem:H}
Let $(X, S)$ be an association scheme with the trivial partition $S=\{{\bf 1}, \sigma\}$ and $F, G: \jmath(X, S) \to \jmath(X, S)$ 
self-homotopy equivalences.  Suppose that $\sharp X \geq 3$ and $F \sim G$. Then $F=G$. 
\end{lem}

\begin{proof} In order to prove the lemma, 
it suffices to show that if there exists a homotopy $H: F \Rightarrow G$, then $F=G$. The homotopy gives rise to the commutative diagram 
$$
\xymatrix@C35pt@R20pt{
{F(i)}  \ar@{->}[r]^{\phi_{F(i)G(i)}} \ar[dr]^{\phi_{F(i)G(j)}}  \ar@{->}[d]_{F(\phi_{ij})} &  {G(i)}  \ar@{->}[d]^{G(\phi_{ij})}\\
{F(j)} \ar@{->}[r]_{\phi_{F(j)G(j)}} &  {G(j)}, }
$$
where $\phi_{ij} = (j, i) \in X\times X$.

Suppose that $F$ is different from $G$. 
Assume further that there exists an object $i$ such that $F(i)=G(i)$. Since $F\neq G$,  it follows that 
$F(j) \neq G(j)$ for some $j$. We see that 
$H(1_i, u) = \phi_{F(i)G(i)}=1_i \in {\bf 1}$ and $H(1_j, u) = \phi_{F(j)G(j)} \in mor(\C)\backslash {\bf 1}$, 
which is a contradiction; see Remark \ref{rem:square}. This implies that $F(j) \neq G(j)$ for any $j$. 

If there exists an element $(i, j)\notin {\bf 1}$ such that $F(i) = G(j)$, 
then $H(\phi_{ij}, u)=\phi_{F(i)G(j)}$ is in ${\bf 1}$ and hence so is $\phi_{F(k)G(l)}$ for any $(k, l)\notin {\bf 1}$. 
This yields that $F(k) = G(l)$ for any $(k, l)\notin {\bf 1}$. Since $\sharp X \geq 2$, 
it follows that $G(1)=F(0)=G(2)$, which is a contradiction. In fact, 
by Lemma \ref{lem:iso} the morphism $G$ is an isomorphism. In consequence, we see 
that $F(i) \neq G(j)$ for any $i$ and $j$ in $X$. Thus, $F(0) \neq G(i)$ for any $i$. The fact enables us to deduce that 
$G$ is not surjective, which is a contradiction. 
This completes the proof.
\end{proof}

\begin{thm}\label{thm:trivial_one}
Let $(X, S)$ be an association scheme with the trivial partition. Then the group 
$h\text{\em aut}(\jmath(X, S))$ is isomorphic to the permutation group of order $\sharp X$ if $\sharp X \geq 3$. 
If $\sharp X = 2$, then $h\text{\em aut}(\jmath(X, S))$ is trivial. 
\end{thm}

\begin{proof}
The result for the case where $\sharp X \geq 3$ follows from Lemmas \ref{lem:iso} and \ref{lem:H}. 

Suppose that $\sharp X = 2$. Let $G$ be the only non-identity isomorphism on $\jmath(X, S)$. 
Then we define a homotopy $H : 1  \Rightarrow G$ by 

\begin{center}
$
\xymatrix@C25pt@R20pt{
{0}  \ar@{->}[r]^{\phi_{01}}  \ar[dr]^{\phi_{01}} \ar@{->}[d]_{id_0} &  {1}  \ar@{->}[d]^{id_1}\\
{0} \ar@{->}[r]_{\phi_{01}} &  {1}, }
$
$
\xymatrix@C25pt@R20pt{
{1}  \ar@{->}[r]^{\phi_{10}}  \ar[dr]^{\phi_{01}} \ar@{->}[d]_{id_1} &  {0}  \ar@{->}[d]^{id_0}\\
{1} \ar@{->}[r]_{\phi_{10}} &  {0}, }
$
$
\xymatrix@C25pt@R20pt{
{0}  \ar@{->}[r]^{\phi_{01}}  \ar[dr]^{id_0} \ar@{->}[d]_{\phi_{01}} &  {1}  \ar@{->}[d]^{\phi_{01}}\\
{1} \ar@{->}[r]_{\phi_{10}} &  {0}, }
$
$
\xymatrix@C25pt@R20pt{
{1}  \ar@{->}[r]^{\phi_{10}}  \ar[dr]^{id_1} \ar@{->}[d]_{\phi_{10}} &  {0}  \ar@{->}[d]^{\phi_{01}}\\
{0} \ar@{->}[r]_{\phi_{01}} &  {1}. }
$
\end{center}
In each square, upper and lower horizontal arrows are in the same element of $S$. 
In the first two squares, the diagonals are in the same element of $S$.  The same condition holds for the second two squares. 
This implies that $H$ is well defined; that is, $H$ is in a morphism in $q\mathsf{ASmd}$; see Remark \ref{rem:square}. We have the result. 
\end{proof}

The following theorem exhibits rigidity of strong homotopy on finite groups.  

\begin{prop}\label{prop:isomorphism}
For a finite group $G$,  every self-homotopy equivalence on a quasi-schemoid of the form 
$\widetilde{S}(\imath G)= (\widetilde{\imath G}, \{\G_s\}_{s\in G})$ is an isomorphism. 
\end{prop}

\begin{proof} 
The set ${\bf 1} := \{1_x\}_{x \in ob(\widetilde{S}(\imath G))}$ is nothing but the element 
$\{ (h, h) \mid h \in G\}$ in the partition of the set of morphisms of the underlying category of 
the quasi-schemoid $\widetilde{S}(\imath G)$. 

Let $F : \widetilde{S}(\imath G) \to \widetilde{S}(\imath G)$ be a self-homotopy equivalence. 
In order to prove the theorem, it suffices to show that $F$ is injective on $mor(\widetilde{S}(\imath G))$. 
By assumption, there exists 
a homotopy inverse $G$ of $F$. Then we have $GF \simeq 1_\C$. We write $\phi$ for $GF$. 
Suppose that $\phi((f, g)) = \phi((f', g'))$ for $(f, g)$ and $(f', g')$ in $mor(\widetilde{S}(\imath G))$. 
Then it follows that $(\phi(f), \phi(g))=(\phi(f'), \phi(g'))$ and the map $\phi(f, f') = (\phi(f), \phi(f'))$ is the identity. 
Assume that $f\neq f'$. 
By the first argument in the proof, we can apply Proposition \ref{prop:mono1} to the morphism $\phi$. 
Thus we see that 
there exist elements $\sigma$ and $\tau$ such that $\tau$ contains a non-identity element and 
$p_{\sigma\tau}^\sigma\neq 0$ or 
$p_{\tau\sigma}^\sigma\neq 0$.  

Suppose that $p_{\tau\sigma}^\sigma\neq 0$, $\sigma = \G_l$ and $\tau= \G_k$. Then we see that 
there exist morphisms 
$(f, g) : g \to f$ and $(h, g) : g \to h$ in $\G_l$ and $(h, f) : f \to h$ in $\G_k$. 
Therefore, it follows that $h^{-1}g = l$, $f^{-1}g = l$ and $h^{-1}f=k$ and hence $\tau=\G_{1_\bullet}$. 
Since $\G_{1_\bullet} = \{(m, m) \mid m \in mor(\G) \}$, each element in $\tau$ is the identity, which is a contradiction. 
The same argument is applicable to the case where $p_{\sigma\tau}^\sigma\neq 0$. 
Thus we see that $f=f'$. We also have  $g\neq g'$ by the same argument above.  
It turns out that $\phi$ is injective on $mor(\widetilde{S}(\imath G))$.
\end{proof}


\begin{ex} For a non-trivial finite group,  
the schemoid $U\jmath S(G)$ is contractible in $\mathsf{Cat}$ but not $\jmath S(G)$ in $q\mathsf{ASmd}$.  
\end{ex}

We consider the group of self-homotopy equivalences on the quasi-schemoid arising from a groupoid via 
the functor $\widetilde{S}( \ )$.

Let $h\text{Aut}((\C, S))$ be the group of the homotopy classes of autofunctors on a quasi-schemoid $(\C, S)$. 
We have a natural map $\eta_{(\C, S)} : h\text{Aut}((\C, S)) \to  h\text{aut}((\C, S)$. 
For a groupoid $\G$, let $\text{Aut}(\G)$ denote the group of autofunctors on $\G$. 
In particular, $\text{Aut}(\imath G)$ for a group $G$ is nothing but the usual automorphism group $\text{Aut}(G)$ 
of $G$.  
 
\begin{thm}\label{thm:haut(groupoid)} 
Let $\G$ be a groupoid which is not necessarily finite. 
Then the functor $\widetilde{S}( \ )$ gives rise to a commutative diagram 
$$
\xymatrix@C35pt@R20pt{
  & h\text{\em aut}(\widetilde{S}(\G)) \\
\text{\em Aut}(\G) \ar@{>->}[ru]^{\widetilde{S}_{*1}} \ar[r]_-{\widetilde{S}_{*2}}^-{} &  h\text{\em Aut}(\widetilde{S}(\G))
\ar[u]_{\eta_{\widetilde{S}(\G)}}
}
$$
in which $\widetilde{S}_{*1}$ is  a monomorphism. Moreover 
$\widetilde{S}_{*2}$ is an isomorphism provided $\G$ is finite. 
\end{thm}

\begin{cor}\label{cor:haut(group)}
Let $G$ be a finite group. Then $h\text{\em aut}(\jmath S(G))\cong \text{\em Aut}(G)$ as a group. 
\end{cor}

\begin{proof} Proposition \ref{prop:isomorphism}, Theorem \ref{thm:haut(groupoid)} and 
the commutativity of the diagram (2.1) give the result. 
\end{proof}

\begin{ex} Since $S({\mathbb Z}/2)$ is the trivial schmeme, it follows from Theorem \ref{thm:trivial_one} that 
$h\text{aut}(\jmath S({\mathbb Z}/2))$ is trivial. On the other hand, 
Corollary \ref{cor:haut(group)} yields that $h\text{aut}(\jmath S({\mathbb Z}/2))$ is isomorphic to the group 
$\text{Aut}({\mathbb Z/2})$ which is trivial. 
\end{ex}

Before proving Theorem \ref{thm:haut(groupoid)}, we consider the homotopy relation $\simeq$ on 
morphisms between quasi-schemoids which come from groupoids. 

\begin{prop}\label{prop:equivalence_relation}
Let $\G$ and $\calH$ be groupoids, which are not necessarily finite. 
Let $\phi, \psi : \widetilde{S}(\G) \to \widetilde{S}(\calH)$ be morphisms of quasi-schemoids. 
Then $\phi$ is homotopic to $\psi$, namely $\phi \simeq \psi$ if and only if there exists a homotopy from $\phi$ to $\psi$. 
\end{prop}

\begin{lem}\label{lem:thin_case} With the same notation as in Proposition \ref{prop:equivalence_relation}, 
there exists a homotopy $L : \phi \Rightarrow \psi$  if and only if
$\psi(j)^{-1} \phi(i) = \psi(l)^{-1}\phi(k)$ for any $(j, i)$ and $(l, k)$ in $mor(\widetilde{\G})$ with $j^{-1}i=l^{-1}k$. 
\end{lem}

\begin{proof}
We recall that in the category $\widetilde{S}(\G)$, $f=(j, i)$ is a unique morphism from $i$ to $j$. 
Suppose that there exists a homotopy $L : \phi \Rightarrow \psi$ 
between morphisms $\phi$ and $\psi$ from $\widetilde{S}(\G)$ to $\widetilde{S}(\calH)$. 
Then for any morphism $f : i \to j$ and $g : k \to l$ in $\widetilde{S}(\G)$, 
we have commutative diagrams in $\widetilde{S}(\calH)$
\begin{center}
$
\xymatrix@C35pt@R20pt{
{\phi(i)}  \ar@{->}[r]^{L(1_i, u)} \ar[rd]^{L(f, u)} \ar@{->}[d]_{\phi(f)} &  {\psi(i)}  \ar@{->}[d]^{\psi(f)}\\
{\phi(j)} \ar@{->}[r]_{L(1_j, u)} &  {\psi(j)}}
$
and 
$
\xymatrix@C35pt@R20pt{
{\phi(k)}  \ar@{->}[r]^{L(1_k, u)} \ar[rd]^{L(g, u)} \ar@{->}[d]_{\phi(g)} &  {\psi(k)}  \ar@{->}[d]^{\psi(g)}\\
{\phi(l)} \ar@{->}[r]_{L(1_l, u)} &  {\psi(l)}.}
$
\end{center}
Observe that 
$1_i = (i, i) \in \G_{1_{s(i)}}$ for any $i$ and that 
$L(f, u)=(\psi(j), \phi(i))$. By definition,  morphisms $f$ and $g$ are in the same element $\G_h$ of $S$ if and only if 
$j^{-1}i=h=l^{-1}k$. Thus if $j^{-1}i=h=l^{-1}k$, then $L(f, u)$ and $L(g, u)$ are in the same element ${\mathcal H}_{h'}$ 
for some $h' \in mor({\mathcal H})$. 
Therefore, we see that $\psi(j)^{-1} \phi(i) = \psi(l)^{-1}\phi(k)$.  

Suppose that $\psi(j)^{-1} \phi(i) = \psi(l)^{-1}\phi(k)$ for any $(j, i)$ and $(l, k)$ in $mor(\widetilde{S}(\G))$ with $j^{-1}i=l^{-1}k$. 
Then the map $L :  \widetilde{S}(\G) \times I \to \widetilde{S}(\calH)$ defined by the squares above is a well-defined homotopy. We have 
$L : \phi \Rightarrow \psi$. 
\end{proof}

\begin{proof}[Proof of Proposition \ref{prop:equivalence_relation}] 
Lemma \ref{lem:thin_case} yields that if there exists a homotopy 
from $\phi$ to $\psi$, then one has a converse homotopy from $\psi$ to $\phi$. 

Suppose that there exist homotopies $L : \phi \Rightarrow \psi$ and $L' : \psi \Rightarrow \eta$. We see that if 
$j^{-1}i = l^{-1}k$, then $\psi(j)^{-1} \phi(i) = \psi(l)^{-1}\phi(k)$. Since $j^{-1}j = l^{-1}l$, it follows that 
$\eta(j)^{-1} \psi(j) = \eta(l)^{-1}\psi(l)$. This allows one to deduce that 
$\eta(j)^{-1} \phi(i) = \eta(l)^{-1}\phi(k)$ if $j^{-1}i = l^{-1}k$. By Lemma \ref{lem:thin_case}, we have a homotopy from $\phi$ to $\eta$.
This completes the proof. 
\end{proof}


\begin{proof}[Proof of Theorem \ref{thm:haut(groupoid)}]
We show that the homomorphism $\widetilde{S}_{*1} : \text{Aut}(\G) \to h\text{aut}(\widetilde{S}(\G))$ defined by 
$\widetilde{S}_{*1}(u)=[\widetilde{S}(u)]$ is a monomophism.  
Since $(\widetilde{S}(u))(i) = u(i)$ by definition, it follows from Proposition \ref{prop:equivalence_relation} and Lemma \ref{lem:thin_case} 
that $u = v$ if $\widetilde{S}(u)\simeq \widetilde{S}(v)$. 
In fact, for any $i$, we see that $u(i)^{-1}v(i) = u(1_{s(i)})^{-1}v(1_{s(i)})= 1_x1_y$ for some $x$ and $y$ in $ob(\G)$. Then $1_x$ and $1_y$ should be composable. 
This yields that $\widetilde{S}_{*1}$ is a monomorphism. 
We define $\widetilde{S}_{*2} : \text{Aut}(\G) \to h\text{Aut}(\widetilde{S}(\G))$ by 
$\widetilde{S}_{*2}(u)=[\widetilde{S}(u)]$. 
It is readily seen that $\eta_{\widetilde{S}(\G)}\circ \widetilde{S}_{*2}=\widetilde{S}_{*1}$. 

Suppose that $\G$ is finite. In order to prove the latter half of the theorem, 
it suffices to show that $\widetilde{S}_{*2}$ is surjective. 

Let $u$ be an element in $\text{Aut}(\widetilde{S}(\G))$.  
We define a self-functor  $u'$ on $\widetilde{S}(\G)$ by 
$$
u'(i) = u(i)u(1_{s(i)})^{-1}
$$
for any $i \in ob(\widetilde{S}(\G))=mor(\G)$. Observe that $u(1_{s(i)})^{-1}$ and $u(i)$ are composable. 
In fact, we have $s(u(i)) = s(u(1_{s(i)}))$ by \cite[Claim 3.3]{K-M}.  

We show that $u'$ is an autofunctor; that is, $u'$ is bijective on $ob(\widetilde{S}(\G)) = mor(\G)$ and for 
any $k \in mor(\G)$, there exists $l(k) \in mor(\G)$ such that 
$u'(\G_k) \subset \G_{l(k)}$. 
Suppose that $u'(i)=u'(j)$. Then $t(u(1_{s(i)}))=s(u'(i))= s(u'(j)) = t(u(1_{s(i)}))$.  
We see that the pair  $(u(1_{s(i)}), u(1_{s(j)}))$ is a morphism in $\widetilde{S}(\G)$ and hence 
$(1_{s(i)}, 1_{s(j)})$ is in $\widetilde{S}(\G)$. Observe that $u$ has the inverse. 
Thus it follows that $s(i)=s(j)$ and $u(i) = u(j)$. We have $i=j$. 
This implies that $u'$ is bijective on $ob(\widetilde{S}(\G))$ because $mor(\G)$ is finite.

Since $u$ is a morphism of quasi-schemoids, it follows that for any $k \in mor(\G)$, there exists $l(k)' \in mor(\G)$ 
such that $u(\G_k) \subset \G_{l(k)'}$. Suppose that $(i, j)$ is in $\G_k$. 
By definition, we have  $i^{-1}j =k$. Then $s(i) = t(k)$ and $s(k) = s(j)$. 
Moreover, it follows that $(u'(i))^{-1}u'(j) = u(1_{s(i)})u(i)^{-1}u(j)u(1_{s(j)})^{-1}= u(1_{t(k)})l(k)'u(1_{s(k)})^{-1}$. 
We can choose the last element as $l(k)$ mentioned above. 
Furthermore, we see that 
the autofunctor $u'$ preserves the set $\widetilde{\G}^\circ =\{1_x,  \mid x \in ob(\G)\}$, 
which is the set of base points of $\widetilde{S}(\G)$; see \cite[Section 3]{K-M}. 

Let $(j, i)$ and $(l, k)$ be morphisms in $\widetilde{\G}$ which are in the same element $\G_h$ for some $h$ in $mor(\G)$.  
Then we see that $j^{-1}i=h=l^{-1}k$ and hence  $s(i) = s(k)$. 
Moreover, since $u$ is a morphism in $q\mathsf{ASmd}$, it follows that there exists $h' \in mor(\G)$ 
such that  $(u(j), u(i))$ and $(u(l), u(k))$ are in the same element $\G_{h'}$; that is, 
$u(j)^{-1}u(i) = h' = u(l)^{-1}u(k)$. Thus we have 
\begin{eqnarray*}
u(j)^{-1}u'(i) = u(j)^{-1}u(i)u(1_{s(i)})^{-1} = u(l)^{-1}u(k)u(1_{s(k)})^{-1} =  u(l)^{-1}u'(k).
\end{eqnarray*}
Then Lemma \ref{lem:thin_case} yields that 
$u$ is homotopy equivalent to $u'$, which is a base points preserving automorphism,    
in $q\mathsf{ASmd}$. Let $(q\mathsf{ASmd})_0$ be the category of quasi-schemoids with base points. 
The result \cite[Corollary 3.5]{K-M} asserts that the functor $\widetilde{S} : \mathsf{Gpd} \to (q\mathsf{ASmd})_0$ 
is fully faithful. This enables us to conclude that $\widetilde{S}_{*2}$ is surjective. 
This completes the proof.  
\end{proof}


\noindent
{\it Acknowledgements.} The author thanks Kentaro Matsuo who pointed out a mistake in the proof of Lemma \ref{lem:H} in a draft of 
this paper. He is grateful for the referee's careful reading of the previous version of this paper.  

\section{Appendix}

We refer the reader to \cite[Section 2]{K-M} for the definition of association schemoids and their category $\mathsf{ASmd}$. 
In this section, 
we consider a homotopy relation in $\mathsf{ASmd}$ with a {\it cylinder} obtained by modifying 
the quasi-schemoid $I=([1], s)=K([1])$ in Definition \ref{defn:Homotopy}. Unfortunately, the result is trivial; see Assertion \ref{assertion:1} below. 
Thus we would need a different cylinder to develop interesting homotopy theory on $\mathsf{ASmd}$. 

Let $t : [1] \to [1]$ be a contravariant functor defined by $t(0) = 1$ and $t(1) = 0$. Then $\widetilde{I}:=([1], s, t)$ is 
an association schemoid. Observe that this is a unique association schemoid structure on the discrete schemoid $I$. 
Let $F, G : (\C, S, T) \to (\D, S', T')$ be morphisms in $\mathsf{ASmd}$. Then it is natural to define a homotopy relation $F\sim G$ 
in $\mathsf{ASmd}$ by replacing the category $q\mathsf{ASmd}$ with $\mathsf{ASmd}$ in Definition \ref{defn:Homotopy}. 
More precisely, we write $F\sim G$ if there exists a morphism $H : (\C, S, T)\times \widetilde{I} \to (\D, S', T')$ in 
$\mathsf{ASmd}$ such that $H: F \Rightarrow G$ or $H: G \Rightarrow F$; see Remark \ref{rem:square}. 

\begin{assertion}\label{assertion:1}
Let $F$ and $G$ be morphisms of association schemoids from $(\C, S, T)$ to $(\D, S', T')$.  
Then $F\sim G$ if and only if $F=G$. 
\end{assertion}

\begin{proof}
In order to prove the assertion, it sufficies to show that if $H: F \Rightarrow G$ for some 
$H : (\C, S, T)\times \widetilde{I} \to (\D, S', T')$ in $\mathsf{ASmd}$, then  $F=G$. 
Since the homotopy $H$ is a morphism of association schemoids,  it follows that $H\circ (T\times t) = T'H$ 
by definition. 
The morphism $H$ is a homotopy from $F$ to $G$. Then $F(f) = H(f, 1_0)$ for any $f \in mor(\C)$. 
This implies that 
$$
T'F(f) = T'H(f, 1_0) = (H\circ (T'\times T)) (f, 1_0) = H(T(f), 1_1) = GT(f) = T'G(f) 
$$
and hence $F(f) = G(f)$ because $(T')^2 = id_{\D}$ by definition. We have the result. 
\end{proof}

\end{document}